\documentclass[11pt,twoside]{amsart}
\usepackage{amsxtra}
\usepackage{color}
\usepackage{amsopn}
\usepackage{amsmath,amsthm,amssymb,arydshln}
\usepackage{mathrsfs,mathtools}
\usepackage{stmaryrd}
\usepackage{hyperref}
\usepackage{enumerate}
\usepackage{xcolor}
\usepackage{pifont}
\usepackage{adjustbox}
\usepackage{tikz}

\newtheorem{theorem}{Theorem}[section]
\newtheorem{corollary}[theorem]{Corollary}
\newtheorem{proposition}[theorem]{Proposition}
\newtheorem{lemma}[theorem]{Lemma}
\newtheorem*{theorem*}{Theorem}
\theoremstyle{definition}
\newtheorem{definition}[theorem]{Definition}
\newtheorem*{notation}{Notation}

\newtheorem{remark}[theorem]{Remark}

\def\sideremark#1{\ifvmode\leavevmode\fi\vadjust{
		\vbox to0pt{\hbox to 0pt{\hskip\hsize\hskip1em
				\vbox{\hsize3cm\tiny\raggedright\pretolerance10000
					\noindent #1\hfill}\hss}\vbox to8pt{\vfil}\vss}}}

\newcommand{\R}{{\mathbb R}}
\renewcommand{\H}{{\mathbb H}}
\newcommand{\C}{{\mathbb C}}

\newcommand{\bC}{\mathbb{C}}

\newcommand{\bR}{\mathbb{R}}

\newcommand{\beq}{\begin{equation}}
\newcommand{\eeq}{\end{equation}}
\renewcommand{\a}{\alpha}
\renewcommand{\b}{\beta}

\renewcommand{\d}{\delta}

\newcommand{\f}{\varphi}
\newcommand{\g}{\gamma}

\renewcommand{\o}{\omega}

\newcommand{\cA}{\mathcal{A}}

\newcommand{\cQ}{\mathcal{Q}}

\newcommand{\cV}{\mathcal{V}}

\newcommand{\psip}{\psi}
\newcommand{\psim}{\widehat{\psi}}
\newcommand{\Wd}{\sigma}
\newcommand{\Oup}{\left[\Omega^{\scriptscriptstyle 1,1}_{\scriptscriptstyle0}(M)\right]}
\newcommand{\Odp}{\left\llbracket\Omega^{\scriptscriptstyle 2,1}_{\scriptscriptstyle0}(M)\right\rrbracket}
\newcommand{\U}{{\mathrm U}}
\newcommand{\SU}{{\mathrm{SU}}}

\newcommand{\SO}{{\mathrm {SO}}}

\newcommand{\GL}{{\mathrm {GL}}}

\newcommand{\G}{{\mathrm G}}
\newcommand{\K}{{\mathrm K}}

\newcommand{\SL}{{\mathrm {SL}}}

\newcommand{\W}{\wedge}

\DeclareMathOperator\End{End}

\DeclareMathOperator\Ad{Ad}
\DeclareMathOperator\ad{ad}

\newcommand{\Ric}{{\rm Ric}}

\newcommand{\Scal}{{\rm Scal}}

\newcommand{\frg}{\mathfrak{g}}

\newcommand{\frh}{\mathfrak{h}}
\newcommand{\frk}{\mathfrak{k}}
\newcommand{\frm}{\mathfrak{m}}

\newcommand{\frsl}{\mathfrak{sl}}

\renewcommand{\gg}{\mathfrak{g}}
\newcommand{\gh}{\mathfrak{h}}
\newcommand{\gk}{\mathfrak{k}}
\newcommand{\gl}{\mathfrak{l}}
\newcommand{\gm}{\mathfrak{m}}

\newcommand{\gt}{\mathfrak{t}}

\newcommand{\gz}{\mathfrak{z}}

\newcommand{\so}{\mathfrak{so}}

\newcommand{\st}{\ |\ }

\newcommand{\diag}{{\rm diag}}

\newcommand{\sst}{\scriptscriptstyle}
\newcommand{\VSHF}{\cV_{\sst\mathrm{SHF}}}

\numberwithin{equation}{section}

\title{Homogeneous symplectic half-flat 6-manifolds}
\author{Fabio Podest\`a and Alberto Raffero}
\subjclass[2010]{53C30, 53C15, 53D05}
\keywords{Homogeneous spaces, symplectic half-flat, Hermitian Ricci tensor}
\thanks{The authors were supported by GNSAGA of INdAM}
\address{Dipartimento di Matematica e Informatica ``U.~Dini'' \\ Universit\`a degli Studi di Firenze\\ Viale Morgagni 67/a\\ 50134 Firenze\\ Italy}
\email{podesta@math.unifi.it, alberto.raffero@unifi.it}

\begin{document}
\begin{abstract}
We consider 6-manifolds endowed with a symplectic half-flat SU(3)-structure and acted on by a transitive Lie group G of automorphisms. 
We review a classical result allowing to show the non-existence of compact non-flat examples. 
In the noncompact setting, we classify such manifolds under the assumption that G is semisimple. 
Moreover, in each case we describe all invariant symplectic half-flat SU(3)-structures up to isomorphism, showing that the Ricci tensor is always Hermitian  
with respect to the induced almost complex structure. This last condition is characterized in the general case. 
\end{abstract}

\maketitle

\section{Introduction}
This is the first of two papers aimed at studying symplectic half-flat 6-manifolds acted on by a Lie group $\G$ of automorphisms. 
Here, we focus on the homogeneous case, i.e., on transitive $\G$-actions, while in a forthcoming paper we shall consider cohomogeneity one actions. 

An $\SU(3)$-structure on a six-dimensional manifold $M$ is given by an almost Hermitian structure $(g,J)$ and a complex volume form $\Psi = \psip+i\psim$ of constant length. 
By \cite{Hit1}, the whole data depend only on the fundamental 2-form $\omega\coloneqq g(J\cdot,\cdot)$ and on the real 3-form $\psip$,  
provided that they fulfill suitable conditions. 

The obstruction for the holonomy group of $g$ to reduce to $\SU(3)$ is represented by the intrinsic torsion, which is encoded in the 
exterior derivatives of $\omega$, $\psip$, and $\psim$ \cite{ChSa}.  
When all such forms are closed, the intrinsic torsion vanishes identically and the $\SU(3)$-structure is said to be {\em torsion-free}. 

In this paper, we focus on 6-manifolds endowed with an $\SU(3)$-structure $(\omega,\psip)$ such that $d\omega=0$ and $d\psip=0$, 
known as {\em symplectic half-flat} in literature (SHF for short). 
These structures are half-flat in the sense of \cite{ChSa}, and their underlying almost Hermitian structure $(g,J)$ is almost K\"ahler.  

Being half-flat, SHF structures can be used to construct local metrics with holonomy contained in $\G_2$ by solving the so-called 
Hitchin flow equations \cite{Hit1}. 
Moreover, it is known that every oriented hypersurface $M$ of a $\G_2$-manifold is endowed with a half-flat $\SU(3)$-structure, which is SHF when  
$M$ is minimal with $J$-invariant second fundamental form~\cite{MarCab}. 
Starting with a SHF 6-manifold $(M,\omega,\psip)$, it is also possible to obtain examples of closed $\G_2$-structures on the Riemannian product $M\times {\mathbb S}^1$, 
and on the mapping torus of a diffeomorphism of $M$ preserving $\omega$ and $\psip$ (see e.g.~\cite{ManTh}). 

In theoretical physics, compact SHF 6-manifolds arise as solutions of type IIA supersymmetry equations \cite{FiUg}.

SHF 6-manifolds were first considered in \cite{DeB}, and then in \cite{DeBTom,DeBTom0}. 
In \cite{DeBTom0}, equivalent characterizations of SHF structures in terms of the Chern connection $\nabla$ were given, showing that $\mathrm{Hol}(\nabla)\subseteq\SU(3)$. 
Moreover, as $\psip$ defines a calibration on $M$ in the sense of \cite{HarLaw}, the authors introduced and studied special Lagrangian submanifolds in this setting.  

In \cite{BedVez}, the Ricci tensor of an $\SU(3)$-structure was described in full generality. 
Using this result, it was proved that SHF structures cannot induce an Einstein metric unless they are torsion-free. 
It is then interesting to investigate the existence of SHF structures whose Ricci tensor has special features. 
By the results in \cite{BlIa}, the Ricci tensor being $J$-Hermitian seems to be a meaningful condition. 
In Proposition \ref{JRic}, we characterize this property in terms of the intrinsic torsion. 

Recently, A.~Fino and the second author showed that SHF structures fulfilling some extra conditions can be used to obtain explicit solutions of the Laplacian $\G_2$-flow 
on the product manifold $M\times{\mathbb S}^1$ \cite{FiRa}. In particular, the class of SHF structures satisfying the required hypothesis includes those having $J$-Hermitian Ricci tensor. 

Most of the known examples of SHF 6-manifolds consist of six-dimensional simply connected Lie groups endowed with a left-invariant SHF structure. 
The classification of nilpotent Lie groups admitting such structures was given in \cite{ConTom}, 
while the classification in the solvable case was obtained in \cite{FMOU}. 
Previously, some examples on unimodular solvable Lie groups appeared in \cite{DeBTom, FrSch,TomVez}. 
Moreover, in \cite{TomVez} a family of non-homogeneous SHF structures on the 6-torus was constructed. 

In the present paper, we look for new examples in the homogeneous setting.  
We first show that compact homogeneous SHF 6-manifolds with invariant SHF structure are exhausted by flat tori (Corollary \ref{inexistencecpt}). 
This result is based on a classical theorem concerning compact almost K\"ahler manifolds acted on transitively by a semisimple automorphism group  \cite{WoGrII}.  
We then focus on the noncompact case $\G/\K$ with $\G$ semisimple.  
We provide a full classification in Theorem \ref{classThm}, showing that only the twistor spaces of $\R\H^4$ and $\C\H^2$ occur. 
Furthermore, we prove that the former admits a unique invariant SHF structure up to homothety, while the latter is endowed precisely with a one-parameter family 
of invariant SHF structures which are pairwise non-homothetic and non-isomorphic. 
We point out that all almost K\"ahler structures underlying the SHF structures in this family share the same Chern connection, which coincides with the canonical connection of the 
homogeneous space. 
Finally, in both cases a representation theory argument allows to conclude that the Ricci tensor is $J$-Hermitian. 

\begin{notation}
Throughout the paper, we shall denote Lie groups by capital letters, e.g.~$\G$, and the corresponding Lie algebras by gothic letters, e.g.~$\gg$.  
\end{notation}

\section{Preliminaries}
\subsection{Stable 3-forms in six dimensions}
A $k$-form on an $n$-dimensional vector space $V$ is said to be {\em stable} if its orbit under the natural action of $\GL(V)$ is open in $\Lambda^k(V^*)$. 
Among all possible situations that may occur (see e.g.~\cite{Hit,Hit1,Rei}), in this paper we will be concerned with stable 3-forms in six dimensions. 

Assume that $V$ is real six-dimensional, and fix an orientation by choosing a volume form $\Omega\in\Lambda^6(V^*)$. 
Then, every 3-form $\rho\in\Lambda^3(V^*)$ gives rise to an endomorphism $S_\rho:V\rightarrow V$ via the identity 
\begin{equation}\label{P}
\iota_v\rho\W\rho\W\eta = \eta(S_\rho(v))\,\Omega, 
\end{equation}
for all $\eta\in\Lambda^1(V^*)$, where $\iota_v\rho$ denotes the contraction of $\rho$ by the vector $v\in V.$  
By \cite{Hit}, $S_\rho^2=P(\rho)\mathrm{Id}_V$ for some irreducible polynomial $P(\rho)$ of degree 4, and $\rho$ is stable if and only if $P(\rho)\neq0$. 
The space $\Lambda^3(V^*)$ contains two open orbits of stable forms defined by the conditions $P>0$ and $P<0$. 
The $\GL^{\sst+}(V)$-stabilizer of a 3-form $\rho$ belonging to the latter is isomorphic to $\SL(3,\C)$. In this case, $\rho$ induces a complex structure 
\begin{equation}\label{Jpsi}
J_\rho:V\rightarrow V,\quad J_\rho\coloneqq \frac{1}{\sqrt{-P(\rho)}}\,S_\rho,
\end{equation}
and a complex $(3,0)$-form $\rho+i\widehat\rho$, where $\widehat\rho\coloneqq J_\rho\rho =\rho(J_\rho\cdot,J_\rho\cdot,J_\rho\cdot)=-\rho(J_\rho\cdot,\cdot,\cdot)$. 
Moreover, the 3-form $\widehat\rho$ is stable, too, and $J_{\widehat\rho}=J_\rho$. 
\begin{remark}
Note that $S_\rho$, $P(\rho)$ and $J_\rho$ depend both on $\rho$ and on the volume form $\Omega$. 
In particular, after a scaling $(\rho,\Omega)\mapsto(c\rho,\lambda\Omega)$, $c,\lambda\in\R\smallsetminus\{0\}$, they transform as follows
\[
\frac{c^2}{\lambda}\, S_\rho,\qquad \frac{c^4}{\lambda^2}\,P(\rho),\qquad \frac{\left|\lambda\right|}{\lambda}\,J_\rho.
\]
Thus, the sign of $P(\rho)$ does not depend on the choice of the orientation.  
\end{remark}
\subsection{Symplectic half-flat 6-manifolds}
Let $M$ be a connected six-dimensional manifold. An $\SU(3)$-structure on $M$ is an $\SU(3)$-reduction of the structure group of its frame bundle.  
By \cite{Hit1}, this is characterized by the existence of a non-degenerate 2-form $\omega\in\Omega^2(M)$ and a stable 3-form $\psi\in\Omega^3(M)$ 
with $P(\psi_x)<0$ for all $x\in M,$ fulfilling the following three properties. First, the {\em compatibility condition}
\begin{equation}\label{compcond}
\omega\W\psi=0,
\end{equation}
which guarantees that $\omega$ is of type $(1,1)$ with respect to the almost complex structure $J\in\End(TM)$ induced by $\psi$ and by the volume form $\frac{\omega^3}{6}$. 
Second, the {\em normalization condition}
\begin{equation}\label{normalization}
\psip\W\psim = \frac23\,\omega^3,
\end{equation}
where $\psim\coloneqq J\psip$. 
Finally, the positive definiteness of the symmetric bilinear form 
\[
g\coloneqq\omega(\cdot,J\cdot). 
\] 
Note that the pair $(g,J)$ is an almost Hermitian structure with fundamental form $\omega$, and that $\psip+i\psim$ is a complex volume form on $M.$

Given an $\SU(3)$-structure $(\omega,\psip)$ on $M,$ we denote by $*:\Omega^k(M)\rightarrow\Omega^{6-k}(M)$ the Hodge operator defined by the Riemannian metric $g$ 
and the volume form $\frac{\omega^3}{6}$, and we indicate by $\left|\,\cdot\,\right|$ the induced pointwise norm on $\Omega^k(M)$.  
When $k=3,4$, the irreducible decompositions of the $\SU(3)$-modules $\Lambda^{k}\left({\R^6}^*\right)$ give rise to the splittings 
\begin{equation}\label{3forms} 
\Omega^{3}(M) = \mathcal{C}^\infty(M)\,\psip \oplus \mathcal{C}^\infty(M)\,\psim \oplus \left\llbracket\Omega^{2,1}_{0}(M)\right\rrbracket 
\oplus \Omega^{1}(M)\W\omega,
\end{equation}
\begin{equation}\label{4forms}
\Omega^{4}(M) = \mathcal{C}^\infty(M)\,\omega^2 \oplus \left[\Omega^{1,1}_{0}(M)\right] \W\omega \oplus\Omega^{1}(M)\W\psip,
\end{equation}
where
\[
\left[\Omega^{1,1}_{0}(M)\right]	\coloneqq	\left\{\kappa\in\Omega^{2}(M)\st \kappa\W\omega^{2}=0,~J\kappa=\kappa\right\}
\]
is the space of primitive 2-forms of type $(1,1)$, and 
\[ 
\left\llbracket\Omega^{2,1}_{0}(M)\right\rrbracket \coloneqq \left\{\f\in\Omega^{3}(M) \st \f\W\omega=0,~\f\W\psi=\f\W\widehat\psi=0 \right\}
\]
is the space of primitive 3-forms of type $(2,1)+(1,2)$ (see e.g.~\cite{BedVez,ChSa}). 

By \cite{ChSa}, the intrinsic torsion of $(\omega,\psip)$ is determined by $d\omega$, $d\psip$, and $d\psim$. 
In particular, it vanishes identically if and only if all such forms are zero. 
When this happens, the Riemannian metric $g$ is Ricci-flat, ${\rm Hol}(g)$ is a subgroup of $\SU(3)$, and the $\SU(3)$-structure is said to be {\em torsion-free}. 

A six-dimensional manifold $M$ endowed with an $\SU(3)$-structure $(\omega,\psip)$ is called {\em symplectic half-flat} (SHF for short) if both $\omega$ and $\psip$ are closed. 
By \cite[Thm.~1.1]{ChSa}, in this case the intrinsic torsion can be identified with a unique 2-form $\Wd\in\Oup$ such that
\begin{equation}\label{SHFeqn}
d\psim=\Wd\W\omega
\end{equation}
(cf.~\eqref{4forms}). 
We shall refer to $\Wd$ as the {\em intrinsic torsion form} of  the {\em SHF structure} $(\omega,\psip)$.  
It is clear that $\Wd$ vanishes identically if and only if the $\SU(3)$-structure is torsion-free. 
When the intrinsic torsion is not zero, the almost complex structure $J$ is non-integrable,  
and the underlying almost Hermitian structure $(g,J)$ is (strictly) almost K\"ahler. 

Since $\Wd$ is a primitive 2-form of type $(1,1)$, it satisfies the identity $*\Wd=-\Wd\W\omega$. 
Using this together with  \eqref{SHFeqn}, 
it is possible to show that $\Wd$ is coclosed, and that its exterior derivative has the following expression with respect to the decomposition \eqref{3forms} of $\Omega^{3}(M)$
\begin{equation}\label{dw2}
d\Wd = \frac{\left|\Wd\right|^2}{4}\psip +\nu,
\end{equation}
for a unique $\nu\in\Odp$ (see e.g.~\cite[Lemma 5.1]{FiRa} for explicit computations).

\section{Symplectic half-flat SU(3)-structures with $J$-Hermitian Ricci tensor}
In this section, we discuss the curvature properties of a SHF 6-manifold $(M,\omega,\psip)$.  
We begin reviewing some known facts from \cite{BedVez}. 

By \cite[Thm.~3.4]{BedVez}, the scalar curvature of the metric $g$ induced by $(\omega,\psip)$ is given by
\begin{equation}\label{ScalSHF}
\mbox{Scal}(g) = -\frac12\left|\Wd\right|^2.
\end{equation}
Therefore, it is zero if and only if the $\SU(3)$-structure is torsion-free. 

The Ricci tensor of $g$ belongs to the space $\mathcal{S}^{2}(M)$ of symmetric 2-covariant tensor fields on $M.$  
The $\SU(3)$-irreducible decomposition of $\mathcal{S}^{2}({\R^6}^*)$ induces the splitting
\[
\mathcal{S}^{2}(M)= C^\infty(M)\,g\oplus\mathcal{S}^{2}_{\sst+}(M)\oplus\mathcal{S}^{2}_{\sst-}(M),
\]
where 
\[
\mathcal{S}^{2}_{\sst+}(M) \coloneqq \left\{h\in\mathcal{S}^{2}(M)\st Jh=h \mbox{ and } {\rm tr}_{g}h=0\right\},\quad 
\mathcal{S}^{2}_{\sst-}(M) \coloneqq \left\{h\in\mathcal{S}^2(M)\st Jh=-h\right\}.
\]
Consequently, we can write 
\[
{\rm Ric}(g) = \frac16\,{\rm Scal}(g)g +  {\rm Ric}^0(g),
\]
and the traceless part ${\rm Ric}^0(g)$ of the Ricci tensor belongs to $\mathcal{S}^2_{\scriptscriptstyle+}(M)\oplus\mathcal{S}^2_{\scriptscriptstyle-}(M)$. 
It follows from \cite[Thm.~3.6]{BedVez} that for a SHF structure
\begin{equation}\label{ric0}
{\rm Ric}^0(g) = \pi_{\sst+}^{-1}\left(\frac14*(\Wd\W\Wd)+\frac{1}{12}\left|\Wd\right|^2\omega\right) + \pi_{\sst-}^{-1}(2\nu),
\end{equation}
where $\nu$ is the $\Odp$-component of $d\Wd$ (cf.~\eqref{dw2}), 
and the maps $\pi_{\sst+}: \mathcal{S}^2_{\sst+}(M)\rightarrow\Oup$ and $\pi_{\sst-}:\mathcal{S}^2_{\sst-}(M)\rightarrow\Odp$ 
are induced by the pointwise $\SU(3)$-module isomorphisms given in \cite[$\S$2.3]{BedVez}. 

Equation \eqref{ric0} together with a representation theory argument allows to show that the Riemannian metric $g$ induced by a SHF structure is Einstein, 
i.e., ${\rm Ric}^0(g)=0$, if and only if the intrinsic torsion vanishes identically \cite[Cor.~4.1]{BedVez}. 
In light of this result, it is natural to ask which distinguished properties $g$ might satisfy. 
Since the almost Hermitian structure $(g,J)$ underlying a SHF structure is almost K\"ahler, the Ricci tensor of $g$ being $J$-Hermitian seems a meaningful condition.   
Indeed, on a compact symplectic manifold $(M,\omega)$, almost K\"ahler structures with $J$-Hermitian Ricci tensor are the critical points of the Hilbert functional 
restricted to the space of all almost K\"ahler structures with fundamental form $\omega$ (see \cite{ApDr,BlIa}).

Using the above decomposition of ${\rm Ric}(g)$, we can show that SHF structures with $J$-Hermitian Ricci tensor are characterized by the expression of $d\Wd$.
\begin{proposition}\label{JRic}
The Ricci tensor of the metric $g$ induced by a SHF structure $(\omega,\psip)$ is Hermitian with respect to the corresponding almost complex structure $J$ if and only if 
\begin{equation}\label{dw2Jric}
d\Wd = \frac{\left|\Wd\right|^2}{4}\psip.
\end{equation} 
When this happens, the scalar curvature of $g$ is constant. 
\end{proposition}
\begin{proof}
The Ricci tensor of $g$ is $J$-Hermitian if and only if it has no component in $\mathcal{S}^2_{\sst-}(M)$. 
By \eqref{ric0}, this happens if and only if $\nu=0$, i.e., if and only if $d\Wd$ is given by \eqref{dw2Jric}. 

Taking the exterior derivative of both sides of \eqref{dw2Jric}, we get $d\left|\Wd\right|^2\W\psip=0$. 
This implies that $\left|\Wd\right|^2$ is constant, since wedging 1-forms by $\psip$ is injective. 
The second assertion follows then from \eqref{ScalSHF}. 
\end{proof}

Examples of SHF 6-manifolds with $J$-Hermitian Ricci tensor include the twistor space of an oriented self-dual Einstein 4-manifold of negative scalar curvature 
(cf.~\cite{DavMus} and \cite[$\S$1.2]{Xu}).

\section{Homogeneous symplectic half-flat 6-manifolds}
In this section, we focus on the homogeneous case. More precisely, we shall consider the following class of SHF 6-manifolds. 
\begin{definition}
A {\em homogeneous symplectic half-flat manifold} is the data of a SHF 6-manifold $(M,\omega,\psip)$ and a connected Lie group $\G$ acting transitively 
and almost effectively on $M$ preserving the SHF structure $(\omega,\psip)$. 
\end{definition}

Since the pair $(g,J)$ induced by $(\omega,\psip)$ is a $\G$-invariant almost K\"ahler structure, 
the homogeneous manifold $M$ is $\G$-equivariantly diffeomorphic to the quotient $\G/\K$, where $\K$ is a compact subgroup of $\G$  \cite[Ch.~I, Cor.~4.8]{KNI}.

In what follows, we review some basic facts on homogeneous symplectic and almost complex manifolds, and then we will focus on invariant SHF structures on 
compact and noncompact homogeneous spaces. 

\subsection{Invariant almost K\"ahler structures on homogeneous spaces}\label{SCH} 
Let $\G/\K$ be a homogeneous space with $\K$ compact. 
It is well-known that there exists an $\Ad(\K)$-invariant subspace $\frm$ of $\frg$ such that $\frg=\frk\oplus\frm$. 
Moreover, there is a natural identification of $T_{[\K]}(\G/\K)$ with $\frm$, and every $\G$-invariant tensor on $\G/\K$ 
corresponds to an $\Ad(\K)$-invariant tensor of the same type on $\frm$, which we will denote by the same letter.

From now on, we assume that $\G$ is semisimple. Recall that in such a case the Cartan-Killing form $B$ of $\frg$ is non-degenerate. 
 
Given a $\G$-invariant symplectic form $\o$ on $\G/\K$, the corresponding $\Ad(\K)$-invariant $2$-form $\o\in \Lambda^2(\gm^*)$ can be written as 
$\o(\cdot,\cdot) = B(D\cdot,\cdot)$, where $D\in \End(\gm)$ is a $B$-skew-symmetric endomorphism. 

Extend $D$ to an endomorphism of $\gg$ by setting $D|_{\gk}\equiv 0$. Then, $d\o=0$ if and only if $D$ is a derivation of $\gg$ (see e.g.~\cite{BFR}). 
Since $\gg$ is semisimple, there exists a unique $z\in \gg$ such that $D=\ad(z)$. 
By the $\Ad(\K)$-invariance of $\o$, $z$ is centralized by $\gk$, and since $\o$ is non-degenerate on $\gm$, the Lie algebra $\gk$ coincides with the centralizer of $z$ in $\gg$.  
Consequently, $\K$ is connected. 

Since $\K$ is compact, there exists a maximal torus ${\rm T}\subseteq \K$ whose Lie algebra $\gt$ contains the element $z$. 
Using the results of \cite[Ch.~IX, $\S$4]{Hel}, a standard argument allows to show that the complexification $\gg^{\sst{\bC}}$ has a Cartan subalgebra $\gh$ given by $\gt^{\sst{\bC}}$.  
We can then consider the root space decomposition $\gg^{\sst{\bC}} = \gh \oplus \bigoplus_{\a\in R}\gg_\a$ with respect to $\frh$, 
where $R$ is the relative root system and $\gg_\a$ is the root space corresponding to the root $\a\in R$. 
For any pair $\a,\b\in R$ satisfying $\a+\b\neq0$, the root spaces $\gg_\a$ and $\gg_\b$ are $B$-orthogonal. 
Moreover, for each $\a\in R$ we can always choose an element $E_\a$ of $\gg_\a$ so that $\gg_{\a}= {\bC} E_\a$, $B(E_\a,E_{-\a}) = 1$, and 
\[
[E_\a,E_\b]=
\left\{
\begin{array}{ll}
N_{\a,\b}E_{\a+\b},	&\mbox{if } \a+\b\in R,	\\
H_\a				&\mbox{if } \b=-\a,		\\
0 				&\mbox{otherwise},
\end{array} \right.
\]
with $N_{\a,\b}\in\R\smallsetminus\{0\}$, and $H_\a\in \gh$ defined as $\a(H) = B(H_\a,H)$ for every $H\in \gh$  (see e.g.~\cite[p.~176]{Hel}). 

Since $\gk$ contains a maximal torus, we have the decompositions $\gk^{\sst{\bC}} = \gh \oplus \bigoplus_{\a\in R_\gk}\gg_\a$ and 
$\gm^{\sst{\bC}} = \bigoplus_{\b\in R_\gm}\gg_\b$, for two disjoint subsets $R_\gk,R_\gm\subset R$ such that 
\[
R = R_\gk\cup R_\gm,\qquad (R_\gk + R_\gk)\cap R \subseteq R_\gk,\qquad (R_\gk + R_\gm)\cap R \subseteq R_\gm.
\]

Let $J\in \End(\gm)$ be an $\Ad(\K)$-invariant complex structure on $\gm$. 
Then, its complex linear extension $J\in \End(\gm^{\sst{\bC}})$ is $\ad(\gh)$-invariant and commutes with the antilinear involution $\tau$ 
given by the real form $\gg$ of $\gg^{\sst{\bC}}$. 
Moreover, the $\ad(\gh)$-invariance implies that $J$ preserves each root space $\gg_\a$, $\a\in R_\gm$,  
and determines a splitting $\gm^{\sst{\bC}} = \gm^{1,0}\oplus\gm^{0,1}$, where 
\[
\gm^{1,0} = \bigoplus_{\b\in R_\gm^{\sst+}}\gg_\b,\qquad \gm^{0,1} = \bigoplus_{\b\in R_\gm^{\sst-}}\gg_\b,
\]
and $R_\gm = R_\gm^{\sst+}\cup R_\gm^{\sst-}$,  $R_\gm^{\sst-} = - R_\gm^{\sst+}$. The full $\Ad(\K)$-invariance is equivalent to 
\[
(R_\gk+R_\gm^{\sst+})\cap R \subseteq R_\gm^{\sst+}.
\]

\subsection{Non-existence of compact non-flat homogeneous SHF 6-manifolds}\label{nocptSHF}
We begin reviewing a general result on compact homogeneous almost K\"ahler manifolds $\U/\K$, 
which was proved in \cite[Thm.~9.4]{WoGrII} for $\U$ semisimple.  
\begin{proposition} \label{cptaK}
A compact homogeneous almost K\"ahler manifold  $(M,g,J)$ is  K\"ahler. 
\end{proposition}
\begin{proof} 
Let $\U$ be a compact connected Lie group acting transitively and almost effectively by automorphisms on $(M,g,J)$, and  let $\o$ be the fundamental form.  
The group $\U$ is (locally) isomorphic to the product of its semisimple part $\G$ and a torus ${\mathrm Z}$, 
and the manifold $M$ splits as a symplectic product $M_1\times {\mathrm Z}$, where $M_1=\G/\K$ and $\K$ is the 
centralizer of a torus in $\G$ (see \cite[$\S$5]{ZwBo}). 
The splitting is also holomorphically isometric, since the tangent spaces to $M_1$ and ${\mathrm Z}$ are inequivalent as $\K$-modules. \par 
Keeping the same notations as in $\S$\ref{SCH}, we recall that when $\gg$ is a compact semisimple Lie algebra, 
then $\overline{E}_\a \coloneqq \tau(E_\a) = - E_{-\a}$, for every root $\a\in R$. 
 
Now, for every $\a\in R_\gm$, we have $E_\a - E_{-\a}\in \gm$ and 
\[
0 <  g(E_\a - E_{-\a}, E_\a - E_{-\a}) = -2 g(E_\a,E_{-\a}). 
\]
Therefore, when $\a\in R_\gm^{\sst+}$
\[
0 < -2 g(E_\a,E_{-\a}) = -2 \o(E_\a,JE_{-\a}) = 2i \o(E_\a,E_{-\a}) = 2i \a(z).
\]
This means that $\a\in R_\gm^{\sst+}$ if and only if $i\a(z) > 0$. 
Hence,  we have that  $(R_\gm^{\sst+}+R_\gm^{\sst+})\cap R \subseteq R_\gm^{\sst+}$.  
This last condition is equivalent to the integrability of $J$ (see e.g.~\cite[(3.49)]{BFR}).
\end{proof} 

An immediate consequence of the previous proposition is the following. 
\begin{corollary}\label{inexistencecpt} 
Let $(M,\omega,\psip)$ be a compact homogeneous SHF 6-manifold. Then, the $\SU(3)$-structure $(\omega,\psip)$ is torsion-free and $M$ is a flat torus.
\end{corollary}
\begin{proof}  Consider the almost K\"ahler structure $(g,J)$ underlying $(\omega,\psip)$. 
By Proposition \ref{cptaK}, the almost complex structure $J$ is integrable. Then, the $\SU(3)$-structure is torsion-free. 
In particular, the metric $g$ is Ricci-flat, and thus flat by \cite{AK}. 
\end{proof}

\subsection{Noncompact homogeneous SHF 6-manifolds} 
Motivated by the result of $\S$\ref{nocptSHF}, we now look for examples of noncompact homogeneous SHF 6-manifolds. 
In particular, assuming that the transitive group of automorphisms $\G$ is semisimple, we shall prove the following classification result.
\begin{theorem}\label{classThm}
Let $(M,\omega,\psip)$ be a noncompact $\G$-homogeneous SHF 6-manifold, and assume that the group $\G$ is semisimple. 
Then, one of the following situations occurs:
\begin{enumerate}[1)]
\item $M = \SU(2,1)/{\mathrm T}^2$, and there exists a 1-parameter family of pairwise non-homothetic and non-isomorphic invariant SHF structures;
\item $M = \SO(4,1)/\U(2)$, and there exists a unique invariant SHF structure up to homothety. 
\end{enumerate} 
Moreover, in both cases the Riemannian metric induced by the SHF structure has $J$-Hermitian Ricci tensor. 
\end{theorem}
\begin{remark}
Observe that the two examples are precisely the twistor spaces of $\mathbb{C H}^2$ and $\mathbb{R H}^4$. 
The existence of a SHF structure on the latter was already known (see e.g.~\cite[$\S$1.2]{Xu}). 
\end{remark}

For the sake of clarity, we divide the proof of Theorem \ref{classThm} into various steps. We begin showing a preliminary lemma. 
\begin{lemma}\label{simplelemma}
Let $(\G/\K,\o,\psip)$ be a homogeneous SHF 6-manifold with $\G$ semisimple. Then, $\G$ is simple. 
\end{lemma}
\begin{proof}
Suppose that $\G$ is not simple. Then, $\gg$ splits as the sum of two non-trivial ideals $\gg=\gg'\oplus\gg''$. 
Since $\gk$ is the centralizer of an element $z\in\gg$, it splits as $\gk = (\gk\cap\gg')\oplus(\gk\cap\gg'')$, and the manifold $\G/\K$ is   
the product of homogeneous symplectic manifolds of lower dimension, say $\G/\K=\G'/\K' \times \G''/\K''$. 
Without loss of generality, we may assume that $\dim(\G'/\K')=2$ and $\dim(\G''/\K'')=4$. 
The tangent space $\gm$ splits as $\gm'\oplus\gm''$, and a simple computation shows that $[\Lambda^3(\gm'\oplus\gm'')]^\K = \{0\}$, since the isotropy representations of 
$\K'$ and $\K''$ have no non-trivial fixed vectors. 
\end{proof}

By the previous lemma, we can focus on the case when the Lie group $\G$ is simple and noncompact. 
Let ${\mathrm L}\subset\G$ be a maximal compact subgroup containing $\K$. Then  $(\G,{\mathrm L})$ is a symmetric pair, and $\K$ is strictly contained in ${\mathrm L}$. 
Indeed, if ${\mathrm L} = \K$, then $(\G,{\mathrm L})$ would be a Hermitian symmetric pair, and every invariant almost complex structure on $\G/{\mathrm L}$ would be integrable. 
In particular, every invariant SHF structure on $\G/{\mathrm L}$ would be torsion-free, hence flat. This contradicts the simplicity of $\G$. 
Moreover, the space ${\mathrm L}/\K$ is symplectic, as $\K$ is the centralizer of a torus in ${\mathrm L}$. Consequently, as $\dim(\G) \geq 6$, we have $\dim(\G/{\mathrm L})=4$.

Therefore, we have to consider the list of symmetric pairs $(\gg,\gl)$ of noncompact type, where $\gg$ is simple, $\gl$ is of maximal rank in $\gg$, 
and $\dim(\gg)-\dim(\gl)=4$. After an inspection of all potential cases in \cite[Ch.~X, $\S$6]{Hel}, 
we are left with two possibilities, which are summarized in Table \ref{ncptSHF}. 

\begin{table}[ht]
\centering
\renewcommand\arraystretch{1.1}
\begin{tabular}{|c|c|c|}
\hline
$\G$				& 	${\mathrm L}$					&	$\K$					 	\\ \hline \hline
$\SU(2,1)$		&	${\mathrm S}(\U(2)\times\U(1))$	&	${\mathrm T}^2$			\\ \hline 
$\SO(4,1)$		&	$\SO(4)$						&	$\U(2)$					\\ \hline    
\end{tabular}
\vspace{0.1cm}
\caption{}\label{ncptSHF}
\end{table}
\renewcommand\arraystretch{1}

We now deal with the two cases separately. \par
\medskip \noindent
{\bf1)} {\boldmath$M  = \SU(2,1)/{\mathrm T}^2$}\\ 
Here $\gg^{\sst{\bC}} = \frsl(3,\bC)$, and we may think of $\gt$ as the abelian subalgebra 
\[
\gt=\left\{\diag(ia,ib,-ia-ib)\in \gg^{\sst{\bC}}\st a,b\in\R\right\}.
\] 
The root system $R$ relative to the Cartan subalgebra $\gt^{\sst{\bC}}$ is given by $\{\pm\a,\pm\b,\pm(\a+\b)\}$. 
Without loss of generality, we assume that $\pm\a$ are the compact roots, i.e., $\gl^{\sst{\bC}} = \gt^{\sst{\bC}} \oplus \gg_\a\oplus \gg_{-\a}$. 
Notice that $\overline{E}_\g=-E_{-\g}$ for a compact root $\g\in R$, while $\overline{E}_\g=E_{-\g}$ when $\g$ is noncompact. 
We can then define the vectors 
\begin{equation}\label{realvect}
v_\g \coloneqq E_\g +  \overline{E}_{\g},\quad w_\g \coloneqq i\left(E_\g - \overline{E}_{\g}\right),\quad \g\in\{\a,\b,\a+\b\}, 
\end{equation}
so that if $\gm_\g \coloneqq \mathrm{span}_\R(v_\g,w_\g)$, we have $\gm = \gm_\a \oplus\gm_\b\oplus\gm_{\a+\b}$.

An invariant symplectic form $\o$ is determined by an element $z\in \gt\smallsetminus\{0\}$, 
and for every root $\g\in R$ the only nonzero components of $\o$ on $\gm^{\sst \C}$ are given by 
\[
\o(E_\g,E_{-\g}) = B([z,E_\g],E_{-\g}) = \g(z).
\]
If we fix $z_{a,b}\coloneqq\diag(ia,ib,-i(a+b))\in\gt$, we have
\[
\renewcommand\arraystretch{1.2}
\begin{array}{rcl}
\o(E_\a,E_{-\a}) 		&=& \a(z_{a,b})~=~i(a-b),\\
\o(E_{\b},E_{-\b}) 		&=& \b(z_{a,b})~=~i(a+2b),\\
\o(E_{\a+\b},E_{-\a-\b}) 	&=& (\a+\b)(z_{a,b})~=~i(2a+b).
\end{array}
\renewcommand\arraystretch{1}
\]

Let $\{E^\g\}_{\g\in R}$ denote the basis of $({\gm^{\sst{\bC}}})^*$ which is dual to the basis given by the root vectors $\{E_\g\}_{\g\in R}$.
Then, we can write
\begin{equation}\label{omegaab}
\omega = i(a-b)\,E^\a \W E^{-\a} + i(a+2b)\,E^\b \W E^{-\b} + i(2a+b)\,E^{\a+\b} \W E^{-\a-\b},
\end{equation}
and the volume form induced by $\omega$ on $\gm^{\sst \C}$ is 
\begin{equation}\label{volomg}
\frac{\omega^3}{6} = i(b-a)(a+2b)(2a+b) E^\a \W E^{-\a} \W E^\b \W E^{-\b} \W E^{\a+\b} \W E^{-\a-\b}. 
\end{equation}
We introduce the real volume form
\begin{equation}\label{Omega}
\Omega\coloneqq iE^\a \W E^{-\a} \W E^\b \W E^{-\b} \W E^{\a+\b} \W E^{-\a-\b}\in\Lambda^6((\gm^{\sst \C})^*).
\end{equation} 
Observe that $\omega^3$ and $\Omega$ define the same orientation if and only if $(b-a)(a+2b)(2a+b)>0$. 

In the next lemma, we describe closed invariant $3$-forms on $M.$
\begin{lemma}\label{Lpsi} Let $\psi\in \Lambda^3(\gm^*)$ be a nonzero $\Ad({\mathrm T}^2)$-invariant 3-form whose corresponding form on $M$ is closed. 
Then, the 3-form $\psi$ on $\gm^{\sst \C}$ can be written as
\begin{equation}\label{psi}
\psi = i q \left( E^{\a}\wedge E^{\b}\wedge E^{-\a-\b} + E^{-\a}\wedge E^{-\b}\wedge E^{\a+\b}\right),
\end{equation}
for a suitable $q\in \R\smallsetminus\{0\}$.
\end{lemma}
\begin{proof}
The invariance of $\psi$ under the adjoint action of the Cartan subalgebra implies that  
\[
(\g_1 + \g_2 + \g_3)(H)\, \psi(E_{\g_1},E_{\g_2},E_{\g_3}) = 0,
\]
for all $\g_1,\g_2,\g_3\in R$ and for all $H\in \gt$. Thus, $\psi$ is completely determined by the values 
\begin{equation}\label{psi1} 
\psi(E_\a,E_\b,E_{-\a-\b}) \coloneqq p+iq,
\end{equation}
and 
\begin{equation}\label{psi2}
\psi(E_{-\a},E_{-\b},E_{\a+\b}) = - \overline{\psi(E_\a,E_\b,E_{-\a-\b})} = -p + iq,
\end{equation}
for suitable $p,q\in\R$. 

Using the Koszul formula for the differential of invariant forms on $\gm^{\sst \C}$, we have
\[
d\psi(X_0,X_1,X_2, X_3) = \sum_{i<j} (-1)^{i+j}\ \psi\left([X_i,X_j]_{\gm^{\sst \C}},X_k,X_l\right),\quad X_0,\ldots, X_3\in \gm^{\sst\C}, 
\]
where $\{i,j\} \cup\{k,l\} = \{0,1,2,3\}$ for each $0\leq i<j\leq3$ and $k<l$. 

By the $\ad(\gt^{\sst{\bC}})$-invariance, we only need to check the values $d\psi\left(E_{\g_1},E_{-\g_1},E_{\g_2},E_{-\g_2}\right)$, with $\g_1,\g_2\in R$. 
From \eqref{psi1}, \eqref{psi2}, and the identity $N_{-\a,-\b} = - N_{\a,\b}$ (cf.~e.g.~\cite[p.176]{Hel}), we get
\begin{eqnarray*}
d\psi(E_\a,E_{-\a},E_\b,E_{-\b}) 	&=&	\psi\left([E_\a,E_\b],E_{-\a},E_{-\b}\right) + \psi\left([E_{-\a},E_{-\b}], E_\a,E_\b\right) \\
							&=&	 N_{\a,\b}\, \psi(E_{-\a},E_{-\b},E_{\a+\b}) + N_{-\a,-\b}\, \psi(E_\a,E_\b,E_{-\a-\b})\\
							&=& 	N_{\a,\b}\, [-p+iq -(p+iq)] \\
							&=&	 -2p N_{\a,\b}. 
\end{eqnarray*}
Similarly, we obtain
\[
d\psip(E_\a,E_{-\a},E_{\a+\b},E_{-\a-\b}) = 2pN_{\a,\b},\quad d\psip(E_\b,E_{-\b},E_{\a+\b},E_{-\a-\b}) = 2pN_{\a,\b}.
\]
Hence, the condition $d\psi=0$ is equivalent to $p=0$. 
\end{proof}
Throughout the following, we will consider a closed invariant $3$-form $\psip$ as in Lemma \ref{Lpsi}. 
The next result proves the compatibility condition \eqref{compcond} and the stability of $\psip$. 
\begin{lemma}\label{Lstab} 
Let $\psi$ be a closed invariant $3$-form on $\gm^{\sst \C}$ as in \eqref{psi}. 
Then, $\psi$ is compatible with every invariant symplectic form $\o$. 
Moreover, $\psi$ is always stable, and it induces an invariant almost complex structure $J\in\End(\gm^{\sst \C})$ such that 
\begin{equation}\label{J} 
J(E_\a) =- i\d_{a,b}\, E_\a,\quad J(E_{\b}) = -i\d_{a,b}\, E_{\b},\quad J(E_{\a+\b}) = i\d_{a,b}\, E_{\a+\b},
\end{equation}
where $\d_{a,b}$ is the sign of $(b-a)(a+2b)(2a+b)$. 
\end{lemma}
\begin{proof} 
First, we observe that $\o\wedge\psi=0$, since there are no non-trivial invariant $5$-forms (or, equivalently, 1-forms) on $\gm$. 
 
In order to check the stability of $\psi$ and compute the almost complex structure induced by it and $\omega^3$,  
we complexify the relation \eqref{P} for the endomorphism $S_\psip$ and we fix the real volume form $\d_{a,b}\,\Omega$ (cf.~\eqref{volomg} and \eqref{Omega}). 
In this way, we obtain a map $S_\psip\in\End(\gm^{\sst\C})$ such that $S_\psip(\gm)\subseteq\gm$ and $S_\psip^2 = P(\psi) \mathrm{Id}$. 
A simple computation shows that for every $\g\in R$
\[
S_\psip(E_\g) = c_\g E_\g,
\]
where
\[
 c_\a = c_\b = - c_{\a+\b} = - \d_{a,b}\,i\,q^2,\quad\mbox{ and }\quad c_{-\g} = - c_\g.
 \]
Consequently, $P(\psi) = - q^4<0$. The expression of $J$ can be obtained from \eqref{Jpsi}. 
\end{proof}
Since $\omega\W\psip=0$, we can consider the $J$-invariant symmetric bilinear form $g \coloneqq \o(\cdot,J\cdot)$.  
It is positive definite if and only if
\begin{eqnarray*}
0	&<&	g(v_\a,v_\a) = 2\,\o(JE_\a,E_{-\a}) = -2i\d_{a,b}\,\o(E_\a,E_{-\a}) = 2\d_{a,b}\,(a-b),\\
0	&<&	g(v_\b,v_\b) = -2\,\o(JE_\b,E_{-\b}) = 2i\d_{a,b}\,\o(E_\b,E_{-\b}) = -2\d_{a,b}\,(a+2b),\\
0	&<&	g(v_{\a+\b},v_{\a+\b}) = -2\,\o(JE_{\a+\b},E_{-\a-\b}) = -2i\d_{a,b}\,\o(E_{\a+\b},E_{-\a-\b}) = 2\d_{a,b}\,(2a+b).
\end{eqnarray*}
Therefore, the set $\cQ$ of admissible real parameters $(a,b)$ can be written as 
$\cQ=\cA\cup (-\cA)$, where 
\[
\cA \coloneqq \left\{(a,b)\left|\ 0< -\frac{a}{2} < b < -2a \right.\right\}.
\]
Note that $\d_{a,b}<0$ and $\d_{-a,-b}>0$ for $(a,b)\in \cA$. 

The last condition we need  is the normalization \eqref{normalization}. Using \eqref{psi} and \eqref{J}, we see that 
\begin{equation}\label{psim1}
\psim = -\d_{a,b}\,q \left( E^{\a}\wedge E^{\b}\wedge E^{-\a-\b} - E^{-\a}\wedge E^{-\b}\wedge E^{\a+\b}\right).
\end{equation}
Thus, 
\[
\psip\wedge\psim = 2\,\d_{a,b}\,q^2\,\Omega.
\]
Combining this identity with \eqref{volomg} and \eqref{normalization} gives 
\begin{equation}\label{normalab}
q^2 = 2 \left|(b-a)(a+2b)(2a+b)\right|,
\end{equation}
which determines $q$ up to a sign. This provides two invariant SHF structures, namely $(\omega,\psip)$ and $(\omega,-\psip)$, which induce isomorphic $\SU(3)$-reductions. 
Hence, we assume $q$ to be positive.

Summing up, for any choice of real numbers $(a,b)\in\cQ$, there is an $\Ad({\mathrm T}^2)$-invariant SHF structure on $\gm$ defined by the 
2-form $\omega$ \eqref{omegaab} and the 3-form $\psip$ \eqref{psi}, with $q>0$ satisfying \eqref{normalab}. 
Moreover, the Ricci tensor of any metric $g$ in this family is $J$-Hermitian. 
Indeed, $\gm$ is the sum of mutually inequivalent ${\mathrm T}^2$-modules, 
and on each module the invariant bilinear form $\Ric(g)$ and the metric $g$ are a multiple of each other. 

Now, we investigate when two invariant SHF structures corresponding to different values of the real parameters $(a,b)\in\cQ$ are isomorphic. \par 

Since the transformation $(a,b)\mapsto (-a,-b)$ maps the $2$-form $\omega$ corresponding to $(a,b)$ into its opposite,  
it leaves the metric $\o(\cdot,J\cdot)$ invariant (cf.~\eqref{J}). 
Note that the standard embedding of $\gg$ into $\frsl(3,\bC)$ (see e.g.~\cite[ p.~446]{Hel}) is invariant under the action of the  conjugation $\theta$ of $\frsl(3,\bC)$ 
with respect to the real form $\frsl(3,\bR)$. The involution $\theta$ preserves $\gt$, and $\theta|_{\gt} =-\mathrm{Id}$. 
The induced map $\hat\theta:M\to M$ is a diffeomorphism with $\hat\theta^*(\o)=-\o$ and $\hat\theta^*(\psip)=\psip$. 
Thus, the SHF structures corresponding to the pairs $(a,b)$ and $(-a,-b)$ are isomorphic, and we can reduce to considering $(a,b)\in\cA$. 

For any nonzero $\lambda\in\R^{\sst+}$, the SHF structures associated with $(a,b)$ and $(\lambda a,\lambda b)$ are homothetic, i.e., 
the defining differential forms and the induced metrics are homothetic. 
Then, we can restrict to a subset of $\cA$ where the volume form is fixed, e.g.
\[
\cV \coloneqq \left\{(a,b)\in \cA\st (b-a)(a+2b)(2a+b)=-1\right\}.
\]  

We now claim that the SHF structures corresponding to the pairs $(a,b)$ and $(b,a)$ in $\cA$ are isomorphic. Indeed, the conjugation 
in $\G$ by the element   
\[
u\coloneqq \left(
\begin{array}{cc:c}
0 & 1 & 0\\
1 & 0 & 0\\ \hdashline
0 & 0 & -1
\end{array}
\right) 
\in \mathrm{S}(\U(2)\times\U(1))
\]
preserves the isotropy $\mathrm{T}^2$ mapping $z_{a,b}$ into $z_{b,a}$. 
Consequently, it induces a diffeomorphism $\phi_{u}:M\rightarrow M,$ which is easily seen to be an isomorphism of the considered SHF structures. 
Therefore, we can further reduce to the set
\[
\VSHF \coloneqq \left\{(a,b)\in\cV\ \left|\ 0< -a \leq b <-2a \right.\right\},
\]
which is represented in Figure \ref{figureset}. 

To conclude our investigation, we prove that the SHF structures corresponding to different points in $\VSHF$ are pairwise non-isomorphic 
by showing that the induced metrics have different scalar curvature. 
From the expression \eqref{psim1} of $\psim$ and the identity $d\psim=\sigma\W\omega$, 
we can determine the intrinsic torsion form $\sigma\in\left[\Lambda^{\sst1,1}_{\sst0}(\gm^*)\right]$ explicitly. 
Then, by \eqref{ScalSHF} we have
\[
\Scal(g) = -\frac{1}{2}|\sigma|^2 = -24\,N_{\a,\b}^2\left(a^2+ab+b^2\right).
\]
Using the method of Lagrange multipliers, it is straightforward to check that the function $\Scal(g)$ subject to the constraint $(b-a)(a+2b)(2a+b)=-1$ 
has a unique critical point at $C = \left(-\frac{1}{\sqrt[3]{2}},\frac{1}{\sqrt[3]{2}}\right)\in\VSHF$. 
Moreover, $\Scal(g)$ is easily seen to be strictly decreasing when the point $(a,b)\in\VSHF$ moves away from $C$.   

\begin{figure}[h!]
\centering
\begin{tikzpicture}
      \draw[->] (-4.5,0) -- (0.5,0); 
      \draw[-] (0.4,0) node[above] {$a$};
      \draw[->] (0,-0.5) -- (0,4.5); 
      \draw[-] (0,4.35) node[right] {$b$};
      \draw[scale=1.5,domain=-2.68:0,dotted,smooth,variable=\a,black] plot(\a, -0.5*\a);
      \draw[scale=1.5,domain=-1.34:0,dotted,smooth,variable=\a,black] plot(\a, -2*\a);
      \draw[scale=1.5,domain=-2.12:0,dotted,smooth,variable=\a,black] plot(\a, -\a);
      \draw[scale=1.5,domain=-1.954:-0.51149,dotted,smooth,variable=\a,black] plot({-1/((2*\a*\a*\a+3*\a*\a-3*\a-2)^(1/3))}, {-\a/((2*\a*\a*\a+3*\a*\a-3*\a-2)^(1/3))});
       \draw[scale=1.5,domain=-1.954:-1,thick,smooth,variable=\a,black] plot({-1/((2*\a*\a*\a+3*\a*\a-3*\a-2)^(1/3))}, {-\a/((2*\a*\a*\a+3*\a*\a-3*\a-2)^(1/3))});
\end{tikzpicture}
\caption{The set $\VSHF$.}\label{figureset}
\end{figure} 

\begin{remark}
Using the properties of the Chern connection $\nabla$ of a homogeneous almost Hermitian space (see e.g.~\cite[$\S$2]{Pod}), 
it is possible to check that the natural operator $\Lambda_{\gm}:\gm\rightarrow \End(\gm)$ associated with $\nabla$ is identically zero for all almost K\"ahler structures  
underlying the SHF structures parametrized by $\VSHF$. Consequently, all $(g,J)$ in this family share the same Chern connection, which 
coincides with the canonical connection of the homogeneous space $\SU(2,1)/\mathrm{T}^2$. 
\end{remark} 
\par

\medskip \noindent
{\bf2)} {\boldmath$M = \SO(4,1)/\U(2)$}\\
In this case, $\gg^{\sst{\bC}}=\so(5,\mathbb C)$. 
We fix the standard maximal abelian subalgebra $\gt$ of the compact real form $\so(5)$ and the corresponding root system $R=\{\pm \a,\pm\b,\pm(\a+\b),\pm(\a+2\b)\}$. 
Without loss of generality, we may choose $R_\gk = \{\pm(\a+2\b)\}$ and $\{\pm\a\}$ as compact roots, and $\{\pm(\a+\b),\pm \b\}$ as noncompact roots. 
Note that $R_\gk\cup\{\pm\a\}$ is the root system of $\gl^{\sst{\bC}}\cong\so(4,\bC)$. The tangent space $\gm$ splits as the sum of two inequivalent $\U(2)$-submodules 
$\gm=\gm_1\oplus\gm_2$, with $\dim_{\R}\gm_1 =2$ and $\dim_{\R}\gm_2 = 4$. 
In particular, if we define the vectors $v_\g,w_\g$ as in \eqref{realvect}, then $\gm_1=\mathrm{span}_\R(v_\a,w_\a)$ 
and $\gm_2=\mathrm{span}_\R(v_\b,w_\b,v_{\a+\b},w_{\a+\b})$.

Any invariant symplectic form $\o$ on $\gm$ is determined by a nonzero element $z$ in the one-dimensional center $\gz$ of $\gk \cong \mathfrak{u}(2)$. 
Since the root $\a+2\b\in R_\gk$ vanishes on $z$, we have $\a(z) = -2\b(z)$. 
Setting $\a(z) = ia$, $a\in \R\smallsetminus\{0\}$, we obtain the following expression for the complexified $\omega$ on $\gm^{\sst\C}$
\[
\omega = ia\,E^\a \W E^{-\a} -\frac12\,ia\,E^\b \W E^{-\b} + \frac12\,ia\,E^{\a+\b} \W E^{-\a-\b},
\]
$\{E^\g\}_{\g\in R}$ being the basis of $({\gm^{\sst{\bC}}})^*$ dual to $\{E_\g\}_{\g\in R}$.

We consider an invariant 3-form $\psi$ on $\gm$ and its complexification on $\gm^{\sst{\bC}}$. 
As in Lemma \ref{Lpsi}, the $\ad(\gt^{\sst{\bC}})$-invariance implies that $\psi$ is completely determined by the value
\[
\psi(E_\a,E_\b,E_{-\a-\b}) \coloneqq p + i q,
\]
and its conjugate 
\[
\psi(E_{-\a},E_{-\b},E_{\a+\b}) = - \overline{\psi(E_\a,E_\b,E_{-\a-\b})} = -p + iq,
\]
for some $p,q\in \bR$. 
In this case, we also have to check the invariance under $\Ad(\U(2))$. This follows  from the vanishing of 
$\psi(E_\a,E_\b,[E_{\a+2\b},E_{-\a-\b}])$, $\psi(E_\a,[E_{-\a-2\b},E_\b],E_{-\a-\b})$, and $\psi(E_{-\a},E_{-\b},[E_{-\a-2\b},E_{\a+\b}])$.

Using the same arguments as in the proofs of Lemma \ref{Lpsi} and Lemma \ref{Lstab}, we can show the following.
\begin{lemma} 
Let $\psi\in \Lambda^3(\gm^*)$ be a nonzero $\Ad(\U(2))$-invariant 3-form whose corresponding form on $M$ is closed. 
Then, the complexified $\psi$ on $\gm^{\sst\C}$ can be written as 
\begin{equation}\label{psiu2}
\psi = i q\, \left( E^{\a}\wedge E^{\b}\wedge E^{-\a-\b} + E^{-\a}\wedge E^{-\b}\wedge E^{\a+\b}\right),
\end{equation}
for a suitable $q\in\R\smallsetminus\{0\}$.
Consequently, $\psi$ is compatible with every invariant symplectic form $\o$, it is always stable, and it induces an invariant almost complex structure $J\in\End(\gm^{\sst\C})$ such that 
\[
J(E_\a) = -i\d_a\, E_\a,\quad J(E_{\b}) = -i\d_a\, E_{\b},\quad J(E_{\a+\b}) = i\d_a\, E_{\a+\b},
\]
where $\d_a$ is the sign of $a$. 
\end{lemma}
The $J$-invariant symmetric bilinear form $g \coloneqq \o(\cdot,J\cdot)$ is positive definite for all $a\in\R\smallsetminus\{0\}$. Indeed
\[
g(v_\a,v_\a) = 2\d_a a,\quad 	g(v_\b,v_\b) = \d_a a,\quad	g(v_{\a+\b},v_{\a+\b}) = \d_a a.
\]
Finally, we observe that 
\[
\o^3 = \frac32a^3\,\Omega,\quad \psip\wedge\psim = 2\d_a\,q^2\,\Omega,
\] 
where $\Omega$ is a real volume form on $\gm^{\sst\C}$ defined as in \eqref{Omega}. 
Therefore, the normalization condition gives 
\[
q^2 = \frac{1}{2}\,\d_a\, a^3.
\]

Summarizing, we have obtained a 1-parameter family of invariant SHF structures on $M$ which are clearly pairwise homothetic. 
As the tangent space $\gm$ has two mutually inequivalent $\U(2)$-submodules, on each module the Ricci tensor of the SHF structure is a multiple of the metric.  
Hence, it is $J$-Hermitian.	
\bigskip

\noindent  {\bf Acknowledgements.} The authors would like to thank Anna Fino for useful comments.


\end{document}